\documentclass[12pt]{amsart}

\usepackage{amssymb}
\usepackage{graphicx}

\usepackage{lineno}

\usepackage{color}

\newtheorem{theorem}{Theorem}

\newtheorem{problem}{Problem}

\newtheorem{them}{Theorem}

\theoremstyle{definition}

\theoremstyle{remark}

\begin{document}

\title[Upper bounds for domination related parameters]{ Upper bounds for domination related parameters in graphs on surfaces}
\author{Vladimir Samodivkin}
\address{Department of Mathematics, UACEG, Sofia, Bulgaria}
\email{vl.samodivkin@gmail.com}
\today

\begin{abstract}
In this paper we give tight upper bounds on the total domination number, 
the weakly connected domination number and the connected domination number of a graph 
in terms of order and Euler characteristic. 
We also present upper bounds for the restrained bondage number,  the total restrained bondage number 
and the restricted edge connectivity of graphs in terms of the  orientable/nonorientable genus and maximum degree. 
\end{abstract}

\maketitle


 \linenumbers

\section{Introduction} \label{Intro}
All graphs considered in this paper are finite, undirected, loopless, and without multiple edges.
We denote the vertex set and the edge set of a graph $G$ by $V(G)$ and $ E(G),$  respectively.   
For a vertex $x$ of $G$,  $N(x)$ denotes the set of all  neighbors of $x$ in $G$ 
   and the degree of $x$ is $\deg(x) = |N(x)|$.  
	The minimum and maximum degree among the vertices of
 $G$ is denoted by $\delta (G)$ and $\Delta (G)$, respectively.
  For $e=xy \in E(G)$, let 
$\xi (e) = deg (x) + deg (y) - 2$ and 
$\xi(G)= min\{\xi(e)$ : $e \in E(G)\}$. The parameter 
$\xi(G)$ is called the minimum edge-degree of $G$.
We let $\left\langle U \right\rangle$ denote the subgraph of $G$ induced
by a subset $U \subseteq V(G)$.
 The girth of a graph $G$, denoted as $g(G)$, is the length of a shortest
cycle in $G$; if $G$ is a forest then $g(G) = \infty$. 
 
An orientable compact 2-manifold $\mathbb{S}_h$ or orientable surface $\mathbb{S}_h$ (see \cite{Ringel}) of genus $h$ is
obtained from the sphere by adding $h$ handles. Correspondingly, a  non-orientable compact
2-manifold $\mathbb{N}_k$ or non-orientable surface $\mathbb{N}_k$ of genus $k$ is obtained from the sphere by
adding $k$ crosscaps. Compact 2-manifolds are called simply surfaces throughout the paper. 
The Euler characteristic is  defined by $\chi(\mathbb{S}_h) = 2 - 2h$, $h\geq  0$, and $\chi(\mathbb{N}_q) = 2 - q$, $q\geq 1$.
A connected graph $G$ is embeddable on a surface $\mathbb{M}$ if it admits a drawing
on the surface with no crossing edges. Such a drawing of $G$ on the surface $\mathbb{M}$ 
is called an embedding of $G$ on $\mathbb{M}$. Notice that there can be many different
embeddings of the same graph $G$ on a particular surface $\mathbb{M}$. 
An embedding of a graph $G$ on surface $\mathbb{M}$ is said to be $2$-cell if every face of the embedding is
homeomorphic to a disc. 
The set of faces of a particular embedding of $G$ on $\mathbb{M}$ is denoted by $F(G)$.
The orientable genus of a graph $G$ is the smallest integer $g = g(G)$ such
that $G$ admits an embedding on an orientable topological surface $\mathbb{M}$ of genus
$g$. The non-orientable genus of $G$ is the smallest integer $\overline{g} = \overline{g}(G)$
 such that $G$ can be embedded on a non-orientable topological surface $\mathbb{M}$ of genus $\overline{g}$.
Note that (a) every embedding of a graph $G$ on $\mathbb{S}_{g(G)}$ is $2$-cell (\cite{yo}), and 
(b)  if a graph $G$ has  non-orientable genus $h$ then $G$ has  $2$-cell embedding on $\mathbb{N}_h$(\cite{pppv}). 
Let a graph $G$ be $2$-cell embedded on a surface $\mathbb{M}$. 
Set $|G| = |V (G)|$, $\|G\| = |E(G)|$, and $f(G) = |F(G)|$. 
The Euler's formula states
\[
|G| - \|G\| + f(G) = \chi(\mathbb{M}).
\]
 

Let $G$ be a non-trivial connected graph and $S \subseteq E(G)$. If $G - S$ is
disconnected and contains no isolated vertices, then $S$ is called a restricted edge-cut
of $G$. The restricted edge-connectivity of $G$, denoted by $\lambda^{\prime}(G)$, 
is defined as the minimum cardinality over all restricted edge-cuts of $G$. 
Besides the classical edge-connectivity $\lambda(G)$, the
parameter $\lambda^{\prime}(G)$ provides a more accurate measure 
of fault-tolerance of networks than the classical
edge-connectivity (see \cite{E}).

A subset $D$ of $V (G)$ is dominating in $G$ if every vertex of $V (G) - D$ has at least
one neighbor in $D$.  The domination number of $G$, 
denoted by $\gamma(G)$, is the size of its smallest dominating set. 
When $G$ is connected, we say $D$ is a connected dominating set if 
 $\left\langle D \right\rangle$ is connected. The connected domination
number of $G$ is the size of its smallest connected dominating set, and is denoted by
$\gamma_c(G)$.
For a connected graph $G$ and any non-empty $S \subseteq V (G)$, $S$ is called
a weakly connected dominating set of $G$ if the subgraph obtained from
$G$ by removing all edges each joining any two vertices in $V (G)-S$ is connected. 
The weakly connected domination number $\gamma_w(G)$ of $G$
 is the minimum cardinality among all weakly connected dominating sets in $G$.

A set $R\subseteq V(G)$ is a restrained dominating set if every vertex not in $R$ is adjacent to
a vertex in $R$ and to a vertex in $V(G) - R$. Every graph $G$ has a restrained dominating set, 
since $R = V(G)$ is such a set. The restrained domination number of $G$, denoted by $\gamma_r (G)$, 
is the minimum cardinality of a restrained dominating set of G.
One measure of the stability of the restrained domination number of $G$ 
under edge removal is the restrained bondage number $b_r(G)$, 
defined in \cite{HP} by Hattingh and Plummer 
as the smallest number of edges whose removal from $G$ results 
in a graph  with larger restrained domination number. 

A set $S\subseteq V (G)$ is a total restrained dominating set, denoted $TRDS$, if
every vertex is adjacent to a vertex in $S$ and every vertex in $V (G)-S$ is also adjacent to a
vertex in $V (G)-S$. The total restrained domination number of $G$, denoted by $\gamma_{tr}(G)$, is the
minimum cardinality of a total restrained dominating set of $G$. 
 Note that any isolate-free graph $G$ has a TRDS, since $V(G)$ is a TRDS.
The total restrained bondage number $b_{tr}(G)$ of a graph $G$ with no isolated vertex, is the
cardinality of a smallest set of edges $E_1\subseteq E(G)$ for which (1) $G-E_1$ has no isolated vertex,
and (2) $\gamma_{tr}(G-E_1) > \gamma_{tr}(G)$. In the case that there is no such subset $E_1$, 
we define $b_{tr}(G) = \infty$.

A labelling $f : V(G) \rightarrow \{0, 1, 2\}$ is a Roman dominating function (or simply an
RDF), if every vertex $u$ with $f (u) = 0$ has at least one neighbour $v$ with $f (v) = 2$.
Define the weight of an RDF $f$ to be $w( f ) = \Sigma_{v\in V(G)} f (v)$. The Roman domination
number of G is $\gamma_R(G) = \min\{w( f ) : f \mbox{ is an RDF}\}$.
The Roman bondage number $b_R(G)$ of a graph $G$ is defined to be the minimum cardinality of all sets
$E \subseteq E(G)$ for which $\gamma_R(G - E) > \gamma_R(G)$. 

The rest of the paper is  organized as follows. 
Section 2 contains known results. 
In section 3 we give tight upper bounds on the total domination number, 
the weakly connected domination number and the connected domination number of a graph 
in terms of order and Euler characteristic.  
In Section 4, we present upper bounds on the restrained bondage number,  the total restrained bondage number 
and the restricted edge connectivity of graphs in terms of the  orientable/nonorientable genus and maximum degree.

\section{Known results}\label{prelimi} 
We make use of the following results in this paper.

\begin{them}[Esfahanian and Hakimi~\cite{EH}] \label{rcon}
If $G$ is a connected graph with at least
four vertices and it is not a star graph, then 
$\lambda^{\prime}(G) \leq \xi (G)$.
\end{them}

\begin{them}[Hatting and Plummer~\cite{HP}]\label{help1}
 If $\delta (G) \geq 2$, then $b_r (G) \leq \xi (G)$.
\end{them}
\begin{them}[Rad, Hasni, Raczek and Volkmann~\cite{rhrv}]\label{help2}
Let $G$ be a connected graph of order $n$, $n\geq 5$. Assume that $G$ has a path $x,y,z$ 
such that $deg(x) > 1$, $deg(z) > 1$ and $G-\{x,y,z\}$ has no isolated vertex. Then
 $b_{tr}(G) \leq deg(x) + deg(y) + deg(z) - 4$. 
\end{them}

\begin{them}[Rad, Hasni, Raczek and Volkmann~\cite{rhrv}]\label{help22a}
 $b_{tr}(K_n) = n-1$ for $n \geq 4$. 
\end{them}

\begin{them}[Rad and Volkmann~\cite{rv}]\label{help2a}
If $G$ is a graph, and $xyz$ a path of length $2$ in $G$, then
$b_R(G)  \leq deg(x) + deg(y) + deg(z) - 3$.
\end{them}


\begin{them} \label{help66}
Let $G$ be a connected graph of order $n$ and size $m$.
\begin{itemize}
\item[(i)]  {\rm (Sanchis~\cite{Sanchist})} 
If $\gamma_t(G) = \gamma_t  \geq 5$, then 
$m \leq \binom{n-\gamma_t +1}{2} + \left\lfloor \frac{\gamma_t}{2}\right\rfloor$. 
\item[(ii)]   {\rm (Sanchis~\cite{Sanchis})} 
 If  $\gamma_w (G) = \gamma_w \geq 3$, then 
$m \leq \binom{n-\gamma_w +1}{2}$. 
\item[(iii)]  {\rm (Sanchis~\cite{Sanchisc})} 
 If  $\gamma_c (G) = \gamma_c \geq 3$, then 
$m \leq \binom{n-\gamma_c +1}{2} +  \gamma_c -1$. 
\end{itemize}
 \end{them}

\begin{them} \label{dghhmThm} {\rm (Dunbar et al.~\cite{dghhm})}
If $G$ is a connected graph with $n \geq 2$ vertices then $\gamma_w (G) \leq n/2$. 
\end{them}

The average  degree $ad(G)$ of a graph $G$ is defined as $ad(G) = 2\|G\|/|G|$. 

\begin{them} \label{hra} {\rm (Hartnell and Rall~\cite{hr})}
For any connected nontrivial graph $G$, there exists a pair of vertices, 
say $u$ and $v$, that are either adjacent or at distance $2$ 
from each other, with the property that $deg(u) + deg(v) \leq 2ad(G)$.
\end{them}

\begin{them}[Samodivkin ~\cite{samajc}]\label{SGZ}
Let $G$ be a connected graph embeddable on a surface $\mathbb{M}$ whose Euler
 characteristic $\chi$ is as large as possible and let $g(G) = g < \infty$.
 Then:
 \[
  ad(G) \leq \frac{2g}{g-2}(1 - \frac{\chi}{|G|}).
 \]
\end{them}

Let 
\[ 
  h_1(x) =  
\begin{cases}
2x + 13 &\text{for $0 \leq x \leq 3$}\\
4x + 7 &\text {for $x \geq 3$}
\end{cases}
, \  \ 
h_2(x) = 
\begin{cases}
8 &\text{for $x=0$}\\
4x + 5 &\text {for $x \geq 1$}
\end{cases}
,
\]
\[
  k_1(x) =  
\begin{cases}
2x + 11 &\text{for $1 \leq x \leq 2$}\\
2x + 9 &\text {for $3 \leq x \leq 5$}\\
2x + 7 &\text {for $x \geq 6$}.
\end{cases}
\ \ and \ \ 
k_2(x) = 
\begin{cases}
8 &\text{for $x=1$}\\
2x+ 5 &\text {for $x \geq 2$}.
\end{cases}
\]
\begin{them}[Ivan\v{c}o \cite{i}] \label{help33} 
If $G$ is a connected graph of orientable genus $g$ and minimum 
degree at least $3$, then $G$ contains an edge $e=xy$ such that 
$deg(x)+deg(y) \leq h_1(g)$.
Furthermore, if $G$ does not contain $3$-cycles, then 
 $G$ contains an edge $e=xy$ such that 
$deg(x)+deg(y) \leq h_2(g)$.
\end{them}

\begin{them}[Jendrol$^{\prime}$ and Tuh\'arsky \cite{jt}] \label{help3} 
If $G$ is a connected graph of minimum 
degree at least $3$ on a nonorientable surface of genus $\overline{g} \geq 1$, 
 then $G$ contains an edge $e=xy$ such that 
$deg(x)+deg(y) \leq k_1(\overline{g})$.
Furthermore, if $G$ does not contain $3$-cycles, then 
$deg(x)+deg(y) \leq k_2(\overline{g})$.
\end{them}


A path $uvw$ is a path of type $(i, j, k)$ if $deg(u)  \leq i$, $deg(v) \leq j$, and $deg(w) \leq k$.
\begin{them}[Borodin, Ivanova, Jensen, Kostochka and Yancey \cite{BIJKY}] \label{help444}
Let $G$ be a planar graph with $\delta(G) \geq 3$. If no $2$ adjacent vertices have degree $3$ 
then $G$ has a $3$-path of one of the following types:
\[
\begin{array}{llllllll}
(3,4,11) & (3,7,5) & (3,10,4) & (3,15,3)  & (4,4,9) & (6,4,8) & (7,4,7) & (6,5,6). 
\end{array}
\]
\end{them}

\begin{them} \label{zzz} 
Let $G$ be a connected graph with $n $ vertices and $q$ edges. 
\begin{itemize}
\item[(i)] {\rm (Ringel~\cite{Ringel}, \rm{Stahl}~\cite{St})} If $G$ is not a tree  then $G$ can be $2$-cell embedded on $\mathbb{N}_{q-n+1}$. 
\item[(ii)] {\rm (Jungerman~\cite{jun})}  If $G$ is a $4$-edge connected,  then $G$ can be $2$-cell embedded on 
$\mathbb{S}_{\left\lfloor \frac{q-n+1}{2}\right\rfloor}$.
\end{itemize}
\end{them}

%

\section{Connected,  weakly connected and total domination}\label{twcd} 

\begin{theorem} \label{total}
Let $G$ be a connected graph of order $n$  and total
domination number $\gamma_t \geq 5$, 
which is $2$-cell embedded  on a surface $\mathbb{M}$. Then:
\begin{itemize}
\item[(i)] $n \geq \gamma_t + (1 + \sqrt{9+8(\left\lceil \gamma_t/2 \right\rceil - \chi(\mathbb{M}))})/2$;
\item[(ii)]  $\gamma_t \leq n - \sqrt{n + 2 - 2\chi(\mathbb{M})}$
 when  $\gamma_t$ is even and $\gamma_t \leq n - \sqrt{n + 3 - 2\chi(\mathbb{M})}$ when $\gamma_t$ is odd.
\end{itemize}
 \end{theorem}
\begin{proof}
Note that $n > \gamma_t \geq 5$ and $\chi(\mathbb{M}) \leq 2$. 
Since $f(G) \geq 1$, Euler's formula implies  $n - \|G\| + 1 \leq \chi(\mathbb{M})$. 
 By Theorem \ref{help66}(i) we have 
 $\|G\| \leq (n-\gamma_t +1)(n-\gamma_t)/2 + \left\lfloor \gamma_t/2 \right\rfloor$. 
 Hence 
 \[
 2\chi(\mathbb{M}) \geq 2n+2 - (n-\gamma_t +1)(n-\gamma_t) -2\left\lfloor \gamma_t/2 \right\rfloor,
 \]   
 or equivalently
\[
n^2 - (2\gamma_t+1)n + \gamma_t^2 - \gamma_t +2\left\lfloor \gamma_t/2 \right\rfloor - 2 + 2\chi(\mathbb{M}) \geq 0, \ \mbox{and} 
\]
\[
\gamma_t^2 - 2n\gamma_t + n^2 - n - \alpha + 2\chi(\mathbb{M}) \geq 0, 
\]
where $\alpha = 2$ when $\gamma_t$ is even and $\alpha = 3$ when $\gamma_t$ is odd. 
Solving these inequalities we respectively obtain 
 the bounds stated in (i) and (ii). 
\end{proof}

Next we show that the bounds in Theorem \ref{total} are tight. 
Let $n$,  $d$ and $t$ be integers such that 
$n = d + 4t + 1$,  $t \geq d \geq 6$ and $d \equiv 2 \textrm{ (mod 4)}$.  
Let us consider any graph $G$ which  has the following form: 
\begin{itemize}
\item[$(P_1)$] $G$ is obtained from $K_{n-d}\cup \frac{d}{2}K_2$ by adding edges between the clique and the graph
$\frac{d}{2}K_2$ in such a way that each vertex in the clique is adjacent to exactly one vertex
in $\frac{d}{2}K_2$ and each component of $\frac{d}{2}K_2$ has at least one vertex adjacent to a vertex in
the clique.
\end{itemize}

Clearly, $|G| = n$, $\gamma_t(G) = d$  and 
$\|G\| = \binom{n-d +1}{2} +  \frac{d}{2}$. 
 Hence $p = (\|G\| - |G| +1)/2 = 4t^2+t + (2-d)/4$ 
is an integer and $G$  can be $2$-cell embedded in $\mathbb{N}_{2p}$(by Theorem \ref{zzz}(i)).  
Now, let  in addition,   $\delta (G) \geq 5$.  Then  
since $G$ is clearly $4$-edge connected, Theorem \ref{zzz}(ii) implies that 
  $G$ can be embedded in $\mathbb{S}_p$. 
     It is easy to see that, in both cases,     the inequalities in Theorem \ref{total} become equalities.
    
  \begin{theorem} \label{upper}
Let $G$ be a connected graph of order $n$ which is $2$-cell embedded  on a surface $\mathbb{M}$. 
 If $\gamma_w (G) = \gamma_w \geq 4$ then 
\begin{equation} \label{nga}
n \geq \gamma_w + (1 + \sqrt{9+8\gamma_w-8\chi(\mathbb{M})})/2, \  \mbox{and}
\end{equation}
\begin{equation} \label{gan}
\gamma_w \leq n + (1 - \sqrt{8n + 9 - 8\chi(\mathbb{M})})/2.
\end{equation}
\end{theorem}
\begin{proof}
Since $f(G) \geq 1$, Euler's formula implies  $n - \|G\| + 1 \leq \chi(\mathbb{M})$. 
  Since $\|G\| \leq (n-\gamma_w +1)(n-\gamma_w)/2$ (by Theorem \ref{help66}(ii)), 
we have $2\chi(\mathbb{M}) \geq 2n - (n-\gamma_w +1)(n-\gamma_w) +2$, or equivalently
\[
n^2 - (2\gamma_w+1)n + \gamma_w^2 - \gamma_w - 2 + 2\chi(\mathbb{M}) \geq 0 \ \mbox{and} 
\]
\[
\gamma_w^2 - (2n+1)\gamma_w + n^2 - n - 2 + 2\chi(\mathbb{M}) \geq 0. 
\]
Solving these inequalities we respectively obtain 
 \eqref{nga} and \eqref{gan}, because 
$n \geq 2\gamma_w$ (by Theorem \ref{dghhmThm}). 
\end{proof}

The bounds in Theorem \ref{upper} are attainable. 
Let $n$,  $d$ and $t$ be integers such that 
$t \geq d \geq 4$ and  $n = d + 4t + i$,  
where $i = 1$ when $d$ is odd, and  $i = 2$ when $d$ is even. 
We consider an arbitrary  graph $G$ which  has the following form: 
\begin{itemize}
\item[$(P_2)$]  $G$ is the union of a clique of $n-d$ vertices, 
 and an independent set of size $d$, such
that each of the vertices in the $(n - d)$-clique is adjacent to exactly one of the
vertices in the independent set, and such that each of these $d$ vertices has at least
one vertex adjacent to it.
\end{itemize}

Obviously, $|G| = n$, $\gamma_w(G) = d$  and 
$\|G\| = \binom{n-d +1}{2}$. 
If $p = (\|G\| - |G| +1)/2$ then $p = 4t^2+t + (1-d)/2$ when $d$ is odd, 
 and $p = 4t^2 +3t +1 - d/2$ when $d$ is even.  
 Hence $p$ is an integer and $G$  can be $2$-cell embedded in $\mathbb{N}_{2p}$, 
  which follows by Theorem \ref{zzz}(i).   
Now, let  in addition,   $\delta (G) \geq 4$.  Then  
since $G$ is clearly $4$-edge connected, 
  $G$ can be embedded in $\mathbb{S}_p$(by Theorem \ref{zzz}(ii)). 
     It is easy to see that, in both cases, we have equalities in \eqref{nga} and  \eqref{gan}.

\begin{theorem} \label{connec}
Let $G$ be a connected graph of order $n$ which is $2$-cell embedded  on a surface $\mathbb{M}$. 
 If $ \gamma_c (G) = \gamma_c \geq 3$ then 
\begin{equation} \label{ngaa}
\gamma_c \leq n - (1 + \sqrt{17 - 8\chi(\mathbb{M})})/2.
\end{equation}
\end{theorem}
\begin{proof}
 Note that $\gamma_c \geq 3$ implies $\gamma_c < n$. 
By Theorem \ref{help66}(iii)  we have $2\|G\| \leq (n-\gamma_c +1)(n-\gamma_c) + 2\gamma_c -2$. 
Hence by Euler's formula 
\[
2\chi(\mathbb{M}) \geq 2n - 2\|G\| +2 
\geq 2n +2 - (n-\gamma_c +1)(n-\gamma_c) -2\gamma_c +2,
\]
 or equivalently
 \[
\gamma_c^2 - (2n-1)\gamma_c + n^2 - n - 4 - \chi(\mathbb{M}) \geq 0
\]
Since $\gamma_c < n$, it immediately follows \eqref{ngaa}. 
\end{proof}

The bound in Theorem \ref{connec} is sharp. 
Let $n$,  $d$ and $t$ be integers such that 
$t \geq d \geq 4$ and  $n = d + t$.  
We consider any  graph $G$ which  has the following form: 
\begin{itemize}
\item[$(P_3)$] $G$ is the union of a clique of $n-d$ vertices, 
 and a path of $d$  vertices, where each
vertex in the clique is adjacent to exactly one of the endpoints of the path, 
and each endpoint has at least one clique vertex adjacent to it.
\end{itemize}

Clearly $|G| = n$, $\gamma_c(G) = d$, $\|G\| = \binom{n-d +1}{2} +  d -1$ and  $k = \|G\| - |G| +1 = \binom{t}{2}$.  
Now Theorem \ref{zzz}(i) implies that 
$G$  can be $2$-cell embedded in $\mathbb{N}_{k}$.  
Finally, it is easy  to check that 
$\gamma_c(G) = n - (1 + \sqrt{17 - 8\chi(\mathbb{N}_{k})})/2$.

In ending this section we note that
 in \cite{Samarxiv}, the present author proved analogous results 
 for the ordinary domination number.

\section{Bondage numbers and restricted edge connectivity}\label{rbrc} 
\begin{theorem} \label{main3} 
Let $G$ be a nontrivial connected graph of orientable genus $g$, non-orientable genus $\overline{g}$ 
and minimum degree at least $3$. Then 
\begin{itemize}
\item[(i)] $\max \{\lambda^{\prime}(G) ,b_r(G)\} \leq \min \{h_1 (g), k_1(\overline{g})\}-2$, and 
\item[(ii)] $\max \{\lambda^{\prime}(G) ,b_r(G)\} \leq \min \{h_2 (g), k_2(\overline{g})\}-2$, 
 provided $G$ does not contain $3$-cycles.
 \end{itemize}
 \end{theorem}
\begin{proof}
(i) By combining  Theorem \ref{rcon} and Theorem \ref{help1} with Theorem \ref{help33} 
we obtain the result.

(ii) Theorem \ref{rcon}, Theorem \ref{help1} and Theorem \ref{help3} together 
 immediately imply the required inequality.  
\end{proof}



\begin{theorem} \label{restrtot}
Let $G$ be a connected graph with $\delta(G) \geq 4$. 
\begin{itemize}
\item[(i)] Then  $b_{tr}(G) \leq \xi(G) + \Delta (G) - 2$. 
\item[(ii)] If $G$ is of orientable genus $g$, then 
                   $b_{tr}(G)  \leq h_1(g)+ \Delta (G) - 4$.
Furthermore, if $G$ does not contain $3$-cycles, then 
$b_{tr}(G) \leq h_2(g)+ \Delta (G) - 4$.
\item[(iii)] If $G$ is  of  nonorientable  genus $\overline{g} $, 
 then $b_{tr}(G) \leq k_1(\overline{g})+ \Delta (G) - 4$.
Furthermore, if $G$ does not contain $3$-cycles, then 
$b_{tr}(G) \leq k_2(\overline{g})+ \Delta (G) - 4$.
\item[(iv)] Then $b_{tr}(G) \leq 2ad(G) + \Delta(G) -4$.
\item[(v)] Let $G$ be embeddable on a surface $\mathbb{M}$ whose Euler
 characteristic $\chi$ is as large as possible and let $g(G) = g$.
 Then:
 \[
  b_{tr}(G) \leq \frac{4g}{g-2}(1 - \frac{\chi}{|G|}) + \Delta(G) -4 \leq -\frac{12\chi}{|G|} + \Delta(G) +8.
 \]
\end{itemize}
\end{theorem}
\begin{proof}
(i) 
Since $\delta (G) \geq 4$,  there is  a path $x,y,z$ in $G$ such that 
$G-\{x,y,z\}$ has no isolated vertices and  $\xi(xy) = \xi (G)$.
Now, by Theorem \ref{help2} we have 
$b_{tr}(G) \leq deg(x) + deg(y) + deg(z) - 4 \leq \xi(G) + deg(z) - 2 
\leq \xi(G) + \Delta (G) -2$.

(ii) Combining (i) and Theorem \ref{help33} we obtain the required.

(iii) The result follows by combining Theorem \ref{help3} and (i). 

(iv)  By Theorem \ref{help22a} we know that $b_{tr}(K_n) = n-1$ whenever $n \geq 4$. 
Hence we may assume $G$ has nonadjacent vertices. 
 Theorem \ref{hra} implies that there are 
$2$ vertices, say $x$ and $y$, that are either adjacent
or at distance 2 from each other, with the property that 
$deg (x) + deg (y) \leq 2ad (G)$. 
 Since $G$ is connected and $\delta (G) \geq 4$, 
 there is a vertex $z$ such that $xyz$ or $xzy$ is a path.  
 In either case by Theorem \ref{help2}   we have 
 $b_{tr}(G) \leq deg(x) + deg(y) + deg(z) - 4 \leq 2ad(G) + \Delta(G) -4$. 
 
 (v) Theorem \ref{SGZ} and (iv) together imply the result. 
\end{proof}

It is well known that the minimum degree of any planar graph is at most $5$. 

\begin{theorem} \label{last}
Let $G$ be a planar graph with minimum degree $\delta(G) = 4+i$, $i \in \{0,1\}$. 
Then $b_{tr}(G) \leq 14-i$.
\end{theorem}
\begin{proof}
There is a path $x,y,z$ such that $deg(x) + deg(y) + deg(z) \leq 18-i$, 
because of Theorem \ref{help444}. 
Since $\delta (G) \geq 4$, $G-\{x,y,z\}$ has no isolated vertices. 
The result now follows by Theorem \ref{help33}.  
\end{proof}

\begin{problem}
Find a sharp constant upper bound for the total restrained bondage number of a planar graph.
\end{problem}

 Akbari, Khatirinejad and Qajar \cite{akq}  recently showed that
  the roman bondage number of any planar graph is not more than $15$. 
 In case when the planar graph has minimum degree $5$, we  improve this bound by $1$.
\begin{theorem} \label{lastl}
Let $G$ be a planar graph with  $\delta(G) = 5$.  
Then $b_{R}(G) \leq 14$.
\end{theorem}
\begin{proof}
By  Theorem \ref{help444} there is a path $x,y,z$ such that $deg(x) + deg(y) + deg(z) \leq 17$. 
Since $b_R(G)  \leq deg(x) + deg(y) + deg(z) - 3$ (Theorem \ref{help2a}), the result follows. 
\end{proof}
Results on the roman bondage number of graphs on surfaces may be found in \cite{samCz}.

\end{document}